\documentclass[11 pt]{article}
\title{
Finiteness of outer automorphism groups of random right-angled Artin groups
}
\author{Matthew B. Day}
\date{June 14, 2011}
\usepackage{amsmath}
\usepackage{amssymb}
\usepackage{amsthm}
\usepackage{epsfig}

\theoremstyle{plain}
\newtheorem{theorem}{Theorem}[section]
\newtheorem{proposition}[theorem]{Proposition}
\newtheorem{lemma}[theorem]{Lemma}
\newtheorem{corollary}[theorem]{Corollary} 
 
\newtheorem{claim}{Claim} 
\newtheorem*{claim*}{Claim} 
\newtheorem{fact}[theorem]{Fact} 

\theoremstyle{definition}

\theoremstyle{remark}
\newtheorem*{remark}{Remark}

\newcommand\co{\colon}
\newcommand\drop{\setminus}
\newcommand\abs[1]{\lvert #1\rvert}

\newcommand\R{\mathbb{R}}
\newcommand\Z{\mathbb{Z}}

\DeclareMathOperator{\GL}{GL}
\DeclareMathOperator{\st}{st}
\DeclareMathOperator{\lk}{lk}
\DeclareMathOperator{\prob}{Prob}
\DeclareMathOperator{\Aut}{Aut}
\DeclareMathOperator{\Out}{Out}
\newcommand\gnp{G(n,p)}

\newcommand\OAG{\Out(A_\Gamma)}
\newcommand\Expect{\mathbb E}

\newcommand{\wU}[2]{\widehat{U}_{(#1,#2)}}

\begin{document}

\maketitle

\begin{abstract}
We consider the outer automorphism group $\mathrm{Out}(A_\Gamma)$ of the right-angled Artin group $A_\Gamma$ of a random graph $\Gamma$ on $n$ vertices in the Erd\H os--R\'enyi model.
We show that the functions  $n^{-1}(\log(n)+\log(\log(n)))$ and $1-n^{-1}(\log(n)+\log(\log(n)))$ bound the range of edge probability functions for which $\mathrm{Out}(A_\Gamma)$ is finite:
if the probability of an edge in $\Gamma$ is strictly between these functions as $n$ grows, then asymptotically $\mathrm{Out}(A_\Gamma)$ is almost surely finite, and if the edge probability is strictly outside of both of these functions, then asymptotically $\mathrm{Out}(A_\Gamma)$ is almost surely infinite.
This sharpens results of Ruth Charney and Michael Farber from their preprint \emph{Random groups arising as graph products}, arXiv:1006.3378v1.
\end{abstract}

\section{Introduction}
Let $\Gamma$ be a simplicial graph with vertex set $V$.
The \emph{right-angled Artin group} $A_\Gamma$ defined by $\Gamma$ is the group with the presentation
\[\langle V |\text{$ab=ba$ for all $a,b\in V$ with $a$ adjacent to $b$}\rangle.\]
Right-angled Artin groups include free groups and free abelian groups and are common objects of study in geometric group theory.
Outer automorphism groups of right-angled Artin groups exhibit great variety: although they include infinite groups such as outer automorphism groups of free groups and $\GL(n,\Z)$, many of them are finite.

The theory of random graphs is a branch of combinatorics initiated by Erd\H os and R\'enyi in a 1959 paper~\cite{ErdosRenyi}.
Since right-angled Artin groups are indexed over graphs, it is natural to ask about the properties of random ones.
Random right-angled Artin groups were studied by Costa--Farber in~\cite{CostaFarber}, and their automorphism groups were specifically studied by Charney--Farber in~\cite{CharneyFarber}.
Charney and Farber showed that under certain conditions, a random right-angled Artin group almost certainly has a finite outer automorphism group; the results of this paper are a sharpening of their results.

\subsection{Background}\label{se:background}
This paper is about two graph-theoretic notions that arise in the study of right-angled Artin groups: \emph{domination} and \emph{star $2$--connectedness}.
Again let $\Gamma$ be a simplicial graph with vertex set $V$ and adjacency relation~$\sim$.
The \emph{star} of a vertex $a$ of $\Gamma$ is the set consisting of $a$ and all vertices adjacent to~$a$:
\[\st(a) = \{a\}\cup\{b\in V| b\sim a\}.\]
A vertex $a\in V$ is a \emph{star-cut-vertex} for $\Gamma$ if the full subgraph $\Gamma\drop\st(a)$ is disconnected.
The graph $\Gamma$ is \emph{star $2$--connected} if it has no star-cut-vertices.
If $\Gamma$ is star $2$--connected, then either $\Gamma$ is connected or $\Gamma$ is the disjoint union of exactly two complete graphs.
For a pair of distinct vertices $a,b\in V$, $a$ \emph{dominates} $b$ in $\Gamma$ if every vertex adjacent to $b$ is adjacent to or equal to $a$, in other words that
\[\st(b)\drop\{b\}\subset \st(a).\]
We write $a>b$ if $a$ dominates $b$, and refer to $(a,b)$ as a \emph{domination pair}.
Note that it is possible for $a$ to dominate $b$ whether $a\sim b$ or not; if $a\sim b$ then $(a,b)$ is an \emph{adjacent domination pair}, and otherwise it is a \emph{non-adjacent domination pair}.
A vertex is \emph{isolated} if it is adjacent to no other vertices, and it is \emph{central} if it is adjacent to all other vertices.
Note that if $a$ is isolated then $b>a$ for any $b$, and if $a$ is central, then $a>b$ for any $b$.
The following examples are instructive: a path with at least four vertices has exactly four domination pairs, and a path with at least five vertices is not star $2$--connected, but a cycle on at least five vertices is star $2$--connected and has no domination pairs.

The presence of domination pairs and star-cut-vertices indicate the existence of infinite order outer automorphisms of right-angled Artin groups.
M. Laurence showed in~\cite{Laurence} that the automorphism group $\Aut A_\Gamma$ of the right-angled Artin group $A_\Gamma$ of a finite graph $\Gamma$ is generated by finitely many automorphisms that fall into four classes: inversions, symmetries, dominated transvections, and partial conjugations.
While inversions and symmetries generate a finite subgroup of $\Aut A_\Gamma$, dominated transvections and partial conjugations are infinite order.
A dominated transvection always has an infinite-order image in the outer automorphism group $\Out A_\Gamma$, but a dominated transvection will only exist if $\Gamma$ has a domination pair.
If $\Gamma$ is star $2$--connected, then every partial conjugation is an inner automorphism; if there is a star-cut-vertex, then there is a partial conjugation whose image in $\Out A_\Gamma$ has infinite order.
We have explained the following fact.
\begin{fact}\label{fa:outfinite}
The group $\OAG$ is finite if and only if $\Gamma$ is star $2$--connected and has no domination pairs.
\end{fact}
For the details of this argument, we refer the reader to \S6 of Charney--Farber~\cite{CharneyFarber}.

To proceed, we must formalize our notion of random graphs.
The \emph{Erd\H os--R\'enyi model} for random graphs is the sequence of probability spaces $\gnp$, where
\begin{itemize}
\item $n$ varies over the positive integers,
\item $p=p(n)$ is a sequence of probability values in $[0,1]$,
\item the underlying set of $\gnp$ is the finite set of all simplicial graphs with vertex set $V$ of cardinality $|V|=n$,
\item each edge occurs with probability $p$ and independently of other edges.
\end{itemize}
It is easy to see that this last condition uniquely determines $\gnp$; it means that $\gnp$ assigns each graph $\Gamma$ with $m$ edges the probability
\[\prob(\Gamma)=p^m(1-p)^{n-m}.\]
Suppose we are given a sequence of probabilities $p$ and a property of graphs $P$.
We say that $\Gamma\in\gnp$ has the property $P$ \emph{asymptotically almost surely} (\emph{a.a.s.}) if the probability that $\Gamma\in\gnp$ has $P$ goes to $1$ as $n\to\infty$.
This model for random graphs and related models are described in detail in Chapter 2 of Bollob\'as~\cite{Bollobas}.

The work in this paper is inspired by the following results of Charney--Farber~\cite{CharneyFarber}:
\begin{theorem}[Charney--Farber]
Suppose $p$ is any function satisfying
\[p(1-p)n-2\ln(n)\to\infty \text{ as } n\to\infty,\]
(for example, if $p$ is constant in $n$ with $0<p<1$), then $\Gamma\in\gnp$ a.a.s.\ has no  domination pairs.
\end{theorem}
\begin{theorem}[Charney--Farber]
Suppose $p$ is constant with respect to $n$ and
\[1-\frac{1}{\sqrt{2}} < p < 1,\]
then $\Gamma\in\gnp$ a.a.s.\ is star $2$--connected.
\end{theorem}

In this paper, we find sharper descriptions of the functions $p$ for which $\Gamma\in\gnp$ a.a.s.\ has no domination pairs and is star $2$--connected.
Further, we show that for $p$ outside of these ranges, the negations of these statements hold a.a.s.

\subsection{Statement of results}

Our two main theorems explain the asymptotically almost sure existence and nonexistence of domination pairs and star-cut-vertices.

\begin{theorem}\label{th:masterdomination}
Let $C>0$ be any fixed constant, and
suppose $p=p(n)$ is a sequence of probability values.
The existence of domination pairs in $\Gamma\in\gnp$, asymptotically almost surely, is summarized as follows:
\begin{itemize}
\item If $pn^2\to 0$, then there are no adjacent domination pairs.
\item If 
\[p<\frac{\log(n)+\log(\log(n)) - \omega(n)}{n}\]
for some sequence $\omega(n)$ with $\omega(n)\to+\infty$, then there are at least $C$ non-adjacent domination pairs.
\item If we still have $p<n^{-1}(\log(n)+\log(\log(n)) - \omega(n))$, but in addition we have $pn^2\to \infty$, 
then there are also at least $C$ adjacent domination pairs.
\item If
\[\frac{\log(n)+\log(\log(n)) + \omega_1(n)}{n} < p < 1- \frac{\log(n)+\log(\log(n)) + \omega_2(n)}{n},\]
for some sequences $\omega_1(n)$, $\omega_2(n)$, both tending to positive infinity, then there are no domination pairs.
\item If 
\[p > 1- \frac{\log(n)+\log(\log(n)) - \omega(n)}{n}\]
for some sequence $\omega(n)$ with $\omega(n)\to+\infty$, then there are at least $C$ adjacent domination pairs.
\item If we still have $p>1-n^{-1}\log(n)+\log(\log(n)) - \omega(n))$, but in addition we have $(1-p)n^2\to \infty$, 
then there are also at least $C$ non-adjacent domination pairs.
\item If $(1-p)n^2\to 0$, then there are no non-adjacent domination pairs.
\end{itemize}
\end{theorem}

\begin{proof}
It is well known that if $pn^2\to 0$, then the probability that $\Gamma\in\gnp$ is the edgeless graph goes to $1$ (since $\Gamma$ has $\binom{n}{2}$ pairs of vertices, the probability that $\Gamma$ is the edgeless graph is $(1-p)^{n(n-1)/2}\sim e^{-pn(n-1)/2}$ as $n\to\infty$).
Adjacent domination pairs require the existence of edges, and therefore we have proven the first item.
The last item follows by a dual argument: if $(1-p)n^2\to 0$, then the probability that $\Gamma$ is the complete graph goes to $1$, but non-adjacent domination pairs require pairs of vertices with no edges between them.

We prove the items that assert existence in Proposition~\ref{pr:dominationexistence} below, and Theorem~\ref{th:dominationnonexistence} below covers the item asserting the nonexistence of domination pairs.
\end{proof}

\begin{theorem}\label{th:masterstarcut}
Suppose $p=p(n)$ is a sequence of probability values.
There are a.a.s.\ no star-cut-vertices in $\Gamma\in\gnp$ if 
\begin{itemize}
\item for some sequence $\omega(n)$ with $\omega(n)\to\infty$,
\[p > \frac{\log(n)+\log(\log(n))+\omega(n)}{n}\text{, and}\]
\item either $n(1-p)\to 0$ or $n(1-p)\to \infty$.
\end{itemize}
Further, if only the first hypothesis holds, then a.a.s.\ 
for any star-cut-vertex $a\in\Gamma$, 
there is at most one component of $\Gamma\setminus\st(a)$ with more than one vertex.
\end{theorem}
The proof of this theorem appears in \S\,\ref{se:summedcounts} below.
If vertices $a,b,c\in V$ form an isolated triangle in the complement graph $\overline{\Gamma}$, then each of $a$, $b$ and $c$ is a star-cut-vertex.
Isolated triangles are only asymptotically forbidden if $np\to\infty$ or $np\to0$.
This fact is explained in Theorem~V.16 of Bollob\'as~\cite{Bollobas}.
In particular, the presence of isolated triangles in $\overline\Gamma$ explains the possibility of star-cut-vertices if the second hypothesis in Theorem~\ref{th:masterstarcut} fails.

The following corollary is the goal of the paper.
\begin{corollary}
If the probability sequence $p$ satisfies
\[\frac{\log(n)+\log(\log(n)) + \omega_1(n)}{n} < p < 1- \frac{\log(n)+\log(\log(n)) + \omega_2(n)}{n},\]
for some sequences $\omega_1,\omega_2$ limiting to $+\infty$,
then $\OAG$ is a.a.s.\ finite for $\Gamma\in\gnp$.

Conversely, if
\[p<\frac{\log(n)+\log(\log(n)) + \omega(n)}{n}\text{ or }p>1-\frac{\log(n)+\log(\log(n)) + \omega(n)}{n}\]
for some $\omega\to+\infty$, then $\OAG$ is a.a.s.\ infinite for $\Gamma\in\gnp$.
\end{corollary}

\begin{proof}
Theorem~\ref{th:masterdomination} explains that there are a.a.s.\ no domination pairs in the first case, and a.a.s.\ there exist some domination pairs in the second case.
Theorem~\ref{th:masterstarcut} implies that a.a.s.\ there are no star-cut-vertices in the first case above.
Then the corollary follows from Fact~\ref{fa:outfinite}.
\end{proof}

\begin{remark}
If $\Gamma$ has isolated vertices then $\OAG$ has a subgroup isomorphic to the automorphism group of a free group, and if $\Gamma$ has central vertices then $\OAG$ has a subgroup isomorphic to a general linear group over the integers.
By a famous theorem of Erd\H os and R\'enyi (see Theorem~\ref{th:connectivitythreshold}), $\Gamma$ will a.a.s.\ have isolated vertices if $p$ is less than $n^{-1}(\log(n)-\omega(n))$ and central vertices if $p$ is greater than $1-n^{-1}(\log(n)-\omega(n))$ for some $\omega\to\infty$.
However, there are two narrow ranges of probability functions, where $p$ or $1-p$ is between $n^{-1}(\log(n)+\omega_1(n))$ and $n^{-1}(\log(n)+\log(\log(n))-\omega_2(n))$ for any $\omega_1,\omega_2\to\infty$, such that $\Gamma$ and $\overline{\Gamma}$ are a.a.s.\ connected but $\OAG$ is a.a.s.\ infinite.
\end{remark}

We end this section with a corollary that gives some insight into the group theory of $\OAG$ in the case that $p$ does not go to zero quickly enough.
\begin{corollary}
If $p>n^{-1}(\log(n)+\log(\log(n))+\omega(n))$ for some $\omega\to\infty$, then a.a.s. $\Out(A_\Gamma)$ is generated by dominated tranvsections, symmetries, and inversions only---partial conjugations are unnecessary.
\end{corollary}

\begin{proof}
It is enough to explain why a partial conjugation can be expressed as a product of dominated transvections under this hypothesis.
As explained in Theorem~\ref{th:masterstarcut}, this hypothesis implies that a.a.s.\ the star of any star-cut-vertex in $\Gamma$ has at most one complementary component with more than one vertex.
Suppose $a\in\Gamma$ is a star-cut-vertex with this property, and $S\subset\Gamma$ is a union of complementary components of $\st(a)$.
The data of $a$ and $S$ determine a partial conjugation automorphism $\alpha\in\Aut(A_\Gamma)$, which is defined on generators of $A_\Gamma$ as follows:
\[\alpha(b)=\left\{\begin{array}{cc} aba^{-1} & b\in S \\ b & b\notin S.\end{array}\right.\]
If $b\in\Gamma$ is an isolated vertex in $\Gamma\setminus\st(a)$, then $a$ dominates $b$ and the dominated transvections multiplying $b$ by $a^{\pm1}$ on the right and on the left all exist; these are the four automorphisms that fix all generators other than $b$, but send $b$ to $a^{\pm1}b$ or $ba^{\pm1}$.
Let $S_1$ be the component of $\Gamma\setminus \st(a)$ with more than one vertex, if it exists.
If $\alpha$ does not fix $S_1$ (meaning that $S_1\subset S$), we can compose $\alpha$ with an inner automorphism to get an automorphism that does fix $S_1$.
In any case, the class of $\alpha$ in $\OAG$ is represented by an automorphism that conjugates certain generators by $a^{\pm1}$ and fixes the rest, and such that those generators that it does not fix are all dominated by $a$.
This automorphism is certainly a product of the dominated transvections given above.
\end{proof}
\subsection{Conventions}
We always use $\Gamma$ to denote a finite graph with vertex set $V$ and edge relation $\sim$.
Between functions, $\sim$ denotes asymptotic unity.
We use the notations $f\in O(g)$ and $f\lesssim g$ to indicate that eventually $f$ is less than a constant multiple of $g$, and we also use $O(f)$ to denote an unknown function asymptotically bounded by $f$.

\section{Domination pairs}
\subsection{Duality of domination pairs}
We will exploit the following connection between adjacent domination and non-adjacent domination.
The \emph{link} $\lk(a)$ of a vertex $a$ is $\st(a)\setminus\{a\}$; then $a>b$ if and only if $\lk(b)\subset\st(a)$.

\begin{lemma}\label{le:adpnadp}
For $a,b\in V$, we have $a>b$ in $\Gamma$ if and only if $b>a$ in the complement graph $\overline \Gamma$.
In particular, $\overline \Gamma$ has as many adjacent domination pairs as $\Gamma$ has non-adjacent ones, and vice versa.
\end{lemma}

\begin{proof}
We add subscripts to our notations for stars and links to make clear which graph we are taking these stars and links in.
Of course, $a>b$ in $\Gamma$ if and only if $\lk_\Gamma(b)\subset\st_\Gamma(a)$.
Note that $\lk_{\overline{\Gamma}}(a)=V\setminus\st_\Gamma(a)$, and $\st_{\overline{\Gamma}}(b)=V\setminus \lk_\Gamma(b)$.
Then $\lk_{\overline{\Gamma}}(a)\subset\st_{\overline{\Gamma}}(b)$, which proves the lemma.
\end{proof}

\subsection{Existence results}
Our existence results follow well-known facts about random graphs by using some straightforward deductions.
The following statement is taken from Bollob\'as~\cite{Bollobas}, Theorem~III.1, and incorporates a comment preceding that theorem.
\begin{theorem}\label{th:manyvalence1}
Let $C>0$ be fixed.
If $p$ is a sequence of probabilities such that $pn^2\to\infty$ and $pn^{3/2}\to 0$, or if $p>\epsilon n^{-3/2}$ for some $\epsilon>0$ and
\begin{equation}\label{eq:v1condition} n(n-1)p(1-p)^{n-2}\to \infty\text{ as }n\to\infty,\end{equation}
then $\Gamma\in\gnp$ a.a.s.\ has at least $C$ vertices of valence $1$.
\end{theorem}

We note an easy corollary of this.
\begin{corollary}\label{co:manyvalence1}
If $pn^2\to\infty$ and $p < n^{-1}(\log(n)+\log(\log(n))-\omega(n))$ for some $\omega(n)$ tending to positive infinity, 
then $\Gamma\in\gnp$ a.a.s.\ has at least $C$ vertices of valence $1$ for any fixed $C$.
\end{corollary}

\begin{proof}
We break the sequence $\{(n,p(n))\}_n$ into three subsequences.
The first one satisfies $pn^{3/2}\to 0$,
the second satisfies $p>\epsilon n^{-3/2}$ for some $\epsilon>0$ and $p<n^{-1}((1/2)\log(n)-\omega_1(n))$ for some $\omega_1(n)$ with $\omega_1(n)\to\infty$,
and the third satisfies $p> (1/4)n^{-1}\log(n)$.
The probability that $\Gamma$ has at least $C$ vertices of valence $1$ goes to $1$ on all three subsequences;
it does so on the first one because it falls under the first clause of Theorem~\ref{th:manyvalence1} and it does so on the other subsequences by the second clause of that theorem, as we now show.
Since $p\to 0$, we know $(1-p)^{n-2}\sim e^{-np}$; then the limit in Equation~\eqref{eq:v1condition} is asymptotically equivalent to $n^2pe^{-np}$.
By substituting our bounds for $p$ in the second subsequence, we get a lower bound on this limit:
\[n^2pe^{-np} > \epsilon n^{1/2} \cdot n^{-1/2}e^{\omega_1(n)} .\]
We do the same for the third subsequence:
\[n^2pe^{-np} > (1/4) n\log(n) \cdot n^{-1}\log(n)^{-1}e^{\omega(n)}.\]
These lower bounds go to infinity, so the theorem applies.
\end{proof}

The next statement is from Bollob\'as~\cite{Bollobas}, Theorem~V.4.
An \emph{isolated edge} is one both of whose endpoints have valence $1$.
\begin{theorem}\label{th:noisoedges}
Fix $C>0$.
If $2np-\log(n)-\log(\log(n))\to\infty$, then $\Gamma\in\gnp$ a.a.s.\ does not have any isolated edges.
\end{theorem}

From these we deduce the following:
\begin{proposition}\label{pr:v1notiso}
Let $C>0$ be fixed.
If $p$ is in the range
\[\frac{\log(n)+\log(\log(n))+\omega_1(n)}{2n} < p < \frac{\log(n)+\log(\log(n))-\omega_2(n)}{n},\]
for some sequences $\omega_1(n),\omega_2(n)$ approaching positive infinity,
then $\Gamma\in\gnp$ a.a.s.\ has at least $C$ vertices of valence $1$ that are not on isolated edges.
\end{proposition}
\begin{proof}
If $\Gamma$ has less than $C$ vertices of valence $1$ that are not on isolated edges, then either some of the vertices of valence $1$ in $\Gamma$ are on isolated edges, or $\Gamma$ has less than $C$ vertices of valence $1$.
Let $R$ denote the event that $\Gamma$ has less than $C$ vertices of valence $1$ that are not on isolated edges, let $S$ denote the event that there are less than $C$ vertices of valence $1$, and let $T$ denote the event that there are some isolated edges.
Then $R\subset S\cup T$.
Any $p$ in the given range certainly satisfies the hypotheses of Theorem~\ref{th:noisoedges}, by the choice of lower bound.
Also, $p$ satisfies the hypotheses of Corollary~\ref{co:manyvalence1}.
So we know $\prob(S)$ goes to $0$ and $\prob(T)$ goes to $0$.
Then since
\[0 \leq \prob(R) \leq \prob(S) + \prob(T)\]
we have $\prob(R)\to 0$.
\end{proof}

We need one more classical result, due to Erd\H os and R\'enyi~\cite{ErdosRenyi}.
A reference is Bollob\'as~\cite[Theorem~V.3, Theorem~VII.3]{Bollobas}.
\begin{theorem}[Erd\H os--R\'enyi]
\label{th:connectivitythreshold}
If 
\[p<\frac{\log(n)-\omega(n)}{n}\]
for some $\omega(n)\to\infty$, then a.a.s.\ $\Gamma\in\gnp$ has at least $C$ isolated vertices for any fixed $C>0$.

If 
\[p>\frac{\log(n)+\omega(n)}{n}\]
for some $\omega(n)\to\infty$, then a.a.s\ $\Gamma\in\gnp$ is connected.
\end{theorem}

\begin{proposition}\label{pr:dominationexistence}
Let $C>0$.
If 
\[p<\frac{\log(n)+\log(\log(n)) - \omega(n)}{n}\]
for some sequence $\omega(n)$ with $\omega(n)\to+\infty$, then a.a.s.\ there are at least $C$ non-adjacent domination pairs in $\Gamma\in\gnp$.
If further, $pn^2\to \infty$, 
then there are also at least $C$ adjacent domination pairs.
Dually, if
\[1-p<\frac{\log(n)+\log(\log(n)) - \omega(n)}{n}\]
for some sequence $\omega(n)$ with $\omega(n)\to+\infty$, then a.a.s.\ there are at least $C$ adjacent domination pairs in $\Gamma\in\gnp$.
Finally, if $(1-p)pn^2\to \infty$ as well, 
then there are also at least $C$ non-adjacent domination pairs.
\end{proposition}

\begin{proof}
It is enough to show the first two statements in the proposition, where $p\to 0$; Then the second two statements, where $p\to 1$, follow dually by Lemma~\ref{le:adpnadp}.

If $b$ is an isolated vertex, then for any other vertex $a$, the pair $(a,b)$ is a non-adjacent domination pair.
If $b$ is a vertex of valence $1$ and $a$ is the vertex $b$ is adjacent to, then $(a,b)$ is an adjacent domination pair.
If $b$ is a vertex of valence $1$ adjacent to a vertex $a$, and $c$ is some third vertex adjacent to $a$, then $(c,a)$ is a non-adjacent domination pair.
So the number of adjacent domination pairs is at least the number of vertices of valence $1$, and the number of non-adjacent domination pairs is at least the number of isolated vertices, plus the number of vertices of valence $1$ not on isolated edges.

If $p$ satisfies the more general hypotheses of the proposition, we break it into at most two subsequences, one which satisfies the hypotheses of Proposition~\ref{pr:v1notiso} and the other which satisfies $np-\log(n)\to -\infty$.
On the first subsequence, the probability of having at least $C$ valence-$1$ vertices not on isolated edges goes to $1$, 
and on the second subsequence, the probability of having at least $C$ isolated vertices goes to $1$ by Theorem~\ref{th:connectivitythreshold}.
So the probability of having at least $C$ non-adjacent domination pairs goes to $1$ on the entire sequence.

If $p$ satisfies the more restrictive hypotheses, then Corollary~\ref{co:manyvalence1} applies and there are a.a.s.\ at least $C$ vertices of valence $1$.
Then a.a.s.\ we also have at least $C$ adjacent domination pairs.
\end{proof}

\subsection{Nonexistence results}
We proceed to count non-adjacent domination pairs; our results on adjacent ones follow using Lemma~\ref{le:adpnadp}.
It is enough to show the first two statements in the proposition, where $p\to 0$, and the second two statements, where $p\to 1$, follow dually by Lemma~\ref{le:adpnadp}.

\begin{proposition}\label{pr:countnadp}
The expected number of non-adjacent domination pairs in $\Gamma$ in $\gnp$ is
\[n(n-1)(1-p)(p+(1-p)^2)^{n-2}.\]
\end{proposition}

\begin{proof}
Let $X\co \gnp\to \Z$ be the random variable with $X(\Gamma)$ equal to the number of pairs $(a,b)$ of distinct vertices in $V$ with $a$ not adjacent to $b$ and $a>b$ in $\Gamma$.
For each pair $(a,b)\in V^2$ with $a\neq b$, we define a random variable
$\widehat X_{(a,b)}\co \gnp\to \Z$ with $\widehat X_{(a,b)}(\Gamma)$ equal to $1$ if $a$ is not adjacent to $b$ and $a>b$, and equal to $0$ otherwise.
The expectation of $\widehat X_{(a,b)}$ is the probability that $(a,b)$ is a non-adjacent domination pair.

Suppose $a,b\in V$ with $a\neq b$.
For each $c\in V\drop \{a,b\}$, let $S_c\subset \gnp$ denote the event that
either $c\sim a$ or both $c\not\sim a$ and $c\not\sim b$.
This is a union of two disjoint events whose probabilities are $p$ and $(1-p)^2$.
So $\prob(S_c)=p +(1-p)^2$.
Since these events involve different edges for different choices of $c$, the events $\{S_c\}_{c\in V\drop\{a,b\}}$ are independent.
The event $\bigcap_{c\in V\drop\{a,b\}}S_c$ is exactly the event that
every vertex adjacent to $b$ is also adjacent to $a$, which is by definition the event that $a>b$.
So in particular,
\[\prob(a>b)=\prob\left(\bigcap_{c\in V\drop\{a,b\}}S_c\right)=(1-p+p^2)^{n-2}.\]

The event that $a$ dominates $b$ involves only the edges from $a$ and $b$ to other vertices; in particular, the event that $a$ is non-adjacent to $b$ is independent of it.
Then the expectation of $\widehat X_{(a,b)}$ is the product of these probabilities:
\[\Expect(\widehat X_{(a,b)})=(1-p)(p+(1-p)^2)^{n-2}.\]
Since $X$ counts the number of non-adjacent domination pairs, and each $X_{(a,b)}$ counts whether $(a,b)$ is such a pair, we have:
\[X = \sum_{a\in V}\sum_{b\in V\drop\{a\}} \widehat X_{(a,b)}.\]
Then by linearity of expectations:
\[\Expect(X)=n(n-1)(1-p)(p+(1-p)^2)^{n-2}.\]
\end{proof}

To show that probability of a domination pair goes to zero in a certain range, we take limits of these expectations and use Markov's inequality.
We will use the following lemma in taking these limits.
In fact, counting star-cut-vertices will involve a closely related limit, so to reuse this lemma, we use a parameter $k$.
\begin{lemma}\label{le:convexbound}
Suppose $k\geq 1$ is an integer and $p=p(n)$ satisfies
\[2\frac{\log(n)+\omega(n)}{n} \leq p\leq 1-(k+1)\frac{\log(n)+\omega(n)}{n},\]
for some sequence $\omega(n)$ that approaches infinity.
Let $F(x,y)$ be defined by
\[F(x,y)=x^{k+1}(y+(1-y)^{k+1})^{x-k-1},\] for suitable $x,y\in\R$.
Then
$\lim_{n\to\infty} F(n,p)=0.$
\end{lemma}

\begin{proof}
We take the second partial derivative 
\[
\begin{split}
\frac{\partial^2 F}{\partial y^2} &= x^{k+1}(x-k-1)(y+(1-y)^{k+1})^{x-k-3}\\ 
&\quad\cdot ((x-k-2)(1-(k+1)(1-y)^k)^2+(k+1)k(1-y)^{k-1}).
\end{split}
\]
For values of $y$ in $[0,1]$ and values of $x$ in $(k+3,\infty)$,
we see that $\partial^2 F/\partial y^2$ is positive, so that $F$ is concave up in its second input (in this range).
Let $a(n)$ be the lower bound for $p$ from the statement, and let $1-b(n)$ be the upper bound.
Then for large enough $n$, we have
\[0\leq F(n,p(n))\leq \max\{F(n,a(n)),F(n,1-b(n))\}.\]

Using the well-known bound $(1+s)^t\leq e^{st}$ (for $t>0$), we see
\[F(n,a(n))\leq n^{k+1}\exp[(n-k-1)(a(n)+(1-a(n))^{k+1})].\]
We may write this bound as
\[n^{k+1}\exp(-kna(n) + O(na(n)^2))\]
as $n\to\infty$,
using the binomial expansion of $(1-a(n))^{k+1}$ and the fact that $a(n)$ is $O(na(n)^2)$.
This is equivalent to $n^{k+1}\cdot n^{-2k}e^{-2\omega(n)}$ as $n\to\infty$, so $F(n,a(n))\to 0$.

Similarly, we have
\[F(n,1-b(n))\leq n^{k+1}\exp[(n-k-1)(1-b(n)+b(n)^{k+1})],\]
which can be written as
\[n^{k+1}\exp(-nb(n)+O(nb(n)^2))\]
as $n\to\infty$.
This is equivalent to $n^{k+1}\cdot n^{-k-1}\exp(-(k+1)\omega(n))$ as $n\to\infty$, so $F(n,1-b(n))\to 0$.
Then $F(n,p)\to 0$ as $n\to\infty$ for any sequence $p(n)$ in the given range.
\end{proof}

\begin{proposition}\label{pr:firstnonexist}
Suppose $\omega(n)$ is a sequence of real numbers tending to positive infinity.
If $p=p(n)$ is a sequence of probability values satisfying
\begin{equation*}
2\frac{\log(n)+\omega(n)}{n}\leq p\leq 1-\frac{\log(n)+\log(\log(n))+\omega(n)}{n},
\end{equation*}
then the probability that $\Gamma$ in $\gnp$ has a non-adjacent domination pair is a.a.s.\ zero.
\end{proposition}

\begin{proof}
Let $X$ be the random variable from the Proposition~\ref{pr:countnadp}.
According to that proposition, we have 
\[\Expect(X)\leq n^2(p+(1-p)^2)^{n-2}.\]
Then by Lemma~\ref{le:convexbound}, with $k=1$, we have that $E(X)\to 0$ as $n\to\infty$ if we assume that
$p\leq 1-2n^{-1}(\log(n)+\omega(n))$.

Next we momentarily assume that
 $p(n)\geq 1-3n^{-1}\log(n)$.
We change variables to use $q=1-p$.
In this case
\[\Expect(X)\leq n^2q(1-q+q^2)^{n-2}\leq n^2q\exp(-nq+O(nq^2)).\]
Then using the bounds on $q$, we have that
\[\Expect(X)\lesssim 3n\log(n)\cdot n^{-1}\log(n)^{-1}e^{-\omega(n)}.\]
This certainly limits to zero.

If $p$ satisfies the more general bounds, we subdivide the sequence $(n,p(n))$ into two subsequences, the first of which satisfies $p(n)\leq 1-2n^{-1}(\log(n)+\omega(n)$, and the second of which satisfies $p(n)\geq 1-3n^{-1}\log(n)$.
Since $E(X)$ goes to zero on both of these subsequences, we have that $E(X)\to 0$ as $n\to\infty$.

Markov's inequality states that for any $\lambda>0$, 
\[\prob(X\geq \lambda)\leq (1/\lambda)\Expect(X).\]
Setting $\lambda=1/2$, we have $\prob(X\geq 1/2)=\prob(X\neq 0)$, which therefore goes to $0$ as $n\to\infty$.
\end{proof}

To tighten the range of probability functions in which non-adjacent domination pairs occur, we consider the following configuration, which we call a \emph{domination diamond}.
This is a quadruple $(a,b,c,d)$ of vertices, all distinct, with $a\sim b\sim c\sim d \sim a$, $a\not\sim c$, $b\not\sim d$, and $a>c$.

\begin{lemma}\label{le:mustbeldd}
If $(a,c)$ is a non-adjacent domination pair, $c$ is not isolated, and nothing dominates $c$ adjacently, then there is a domination diamond $(a,b,c,d)$ in $\Gamma$.
\end{lemma}

\begin{proof}
Of course $\lk(c)$ is nonempty.
If the induced subgraph on $\lk(c)$ is a complete graph, then any element of $\lk(c)$ adjacently dominates $c$.
So there are two vertices $b,d\in\lk(c)$ with $b\not\sim d$.
Since $b,d\sim c$ and $a>c$, we know $b,d\sim a$.
\end{proof}

\begin{proposition}\label{pr:noldd}
If $p\to 0$ and $np\to\infty$ as $n\to\infty$, then a.a.s.\ $\Gamma\in\gnp$ has no domination diamonds.
\end{proposition}

\begin{proof}
Let $\widehat W_{(a,b,c,d)}$ be the random variable which is $1$ if $(a,b,c,d)$ is a domination diamond and $0$ otherwise.
Then
\[\Expect(\widehat W_{(a,b,c,d)})=p^4(1-p)^2(p+(1-p)^2)^{n-4},\]
since the mandated edges and non-edges among $(a,b,c,d)$ are given, and for any fifth vertex $e$, we must have $e$ adjacent to $c$ or not adjacent to $c$ or $a$.
Setting $W$ equal to the sum of all $\widehat W_{(a,b,c,d)}$ over all choices of $(a,b,c,d)$, we have that the random variable $W$ counts the number of domination diamonds.
By additivity:
\[\Expect(W)=\frac{n!}{(n-4)!}p^4(1-p)^2(p+(1-p)^2)^{n-4}\]
Since $p\to 0$, we know
\[(p+(1-p)^2)^{n-4}\sim \exp(-np).\]
Then
\[\Expect(W)\sim (np)^4\exp(-np).\]
Since the function $t\mapsto t^4e^{-t}$ converges to $0$ as $t\to\infty$, our hypothesis that $np\to\infty$ forces $\Expect(W)$ to converge to $0$ as $p\to\infty$.
The proposition follows immediately by Markov's inequality.
\end{proof}

\begin{theorem}\label{th:dominationnonexistence}
Suppose $\omega_1(n), \omega_2(n)$ are sequences with $\omega_1(n),\omega_2(n)\to+\infty$ as $n\to\infty$.
If $p$ satisfies:
\[\frac{\log(n)+\log(\log(n))+\omega_1(n)}{n}<p<1-\frac{\log(n)+\log(\log(n))+\omega_2(n)}{n}\]
then $\Gamma\in\gnp$ a.a.s.\ has no domination pairs.
\end{theorem}

\begin{proof}
By Lemma~\ref{le:adpnadp}, it is enough to show that a.a.s.\ there are no non-adjacent domination pairs.
We break $p$ into two subsequences, one where Proposition~\ref{pr:firstnonexist} applies, and one where $p\to 0$.
Then it is enough to show that the theorem holds if we assume $p\to 0$ as $n\to\infty$.

Let $A$ be the event that there is some adjacent domination pair,
$B$ the event that there is some non-adjacent domination pair,
$C$ the event that there is some isolated vertex, and
$D$ the event that there is some domination diamond.
Proposition~\ref{pr:firstnonexist} and Lemma~\ref{le:adpnadp} tell us that $\prob(A)\to 0$ in this range, since replacing $p$ with $(1-p)$ puts the probability function in the range in which there are no non-adjacent domination pairs.
Theorem~\ref{th:connectivitythreshold} implies that $\prob(C)\to 0$.
Further, Proposition~\ref{pr:noldd} tell us that $\prob(D)\to 0$.

Lemma~\ref{le:mustbeldd} implies that for every non-adjacent domination pair ($a,c)$, one of the following holds: (1) $c$ is isolated, (2) there is a domination diamond $(a,b,c,d)$ or (3) something adjacently dominates $c$.
In other words, $B\subset A\cup C\cup D$.
Then
\[0<\prob(B)<\prob(A)+\prob(C)+\prob(D),\]
and therefore $\prob(B)\to 0$.
Of course the theorem then follows from Markov's inequality.
\end{proof}

\section{Star 2-connectedness}
\subsection{Star separations}

A subset $S\subset V$ is a \emph{separation} of $\Gamma$ if $S\neq\varnothing$, $S\neq V$ and there are no edges in $\Gamma$ from $S$ to $V\drop S$.
Define a \emph{star separation} of $\Gamma$ to be a pair $(a,S)$ with $a\in V$ and $S\subset V\drop \st(a)$, such $S$ is a separation of $\Gamma\drop\st(a)$.
Call a star separation a star $k$-separation if $\abs{S}=k$.
A star separation $(a,S)$ is \emph{proper} if $S$ is not a separation of $\Gamma$.
Given a separation $S$, there is a star separation $(a,S)$ only if there is a vertex $a\in V\drop S$ such that $\st(a)\neq V\drop S$.
Of course, $\Gamma$ is star $2$-connected if and only if it has no star separations.

If $(a,\{b\})$ is a star separation, then $(a,b)$ is a non-adjacent domination pair.
However, the converse does not hold.
If $V=\st(a)\sqcup\{b\}$, then $(a,b)$ is a non-adjacent domination pair, but $(a,\{b\})$ is not a star separation.
However, this means that in sparse graphs with many vertices, non-adjacent domination pairs are practically the same thing as star $1$-separations.
This explains the similarities in the functions that describe the expected number of each.

\subsection{Counting small star separations}

Our first result on star-separations shows that for $k$ not depending on $n$, star $k$-separations asymptotically almost certainly do not occur in a wide range of probabilities.
\begin{proposition}\label{pr:smallstarseps}
Let $p=p(n)$ be a sequence of probabilities and let $k\geq 1$ be fixed.
Suppose 
\[p\geq \frac{\log(n)+(2/k)\log(\log(n))+\omega(n)}{n}\]
for some $\omega(n)$ approaching positive infinity.
Further suppose that either $k\geq2$ or else that $n(1-p)\to 0$ or $n(1-p)\to\infty$.
Then a.a.s. $\Gamma\in\gnp$ has no star $k$-separations.
\end{proposition}
The proof appears after the next two lemmas.
We would like to proceed by fixing $k$ and counting directly the number of star $k$-separations.
However, the random variable that counts star $k$-separations has a problem:  it turns out that there is a range of probabilities where the probability of a star $k$-separation existing goes to zero even though the expected number of star separations goes to infinity.
We get better bounds by counting proper star $k$-separations instead.
\begin{lemma}\label{le:countss}
Let $k>0$ be an integer and let $U_k$ be the random variable on $\gnp$ that counts the number of proper star $k$-separations.
Then the expectation of $U_k$ is
\[
\begin{split}
\Expect(U_k) & = n\binom{n-1}{k}(1-p)^k\\
&\quad \cdot
\big[(p+(1-p)^{k+1})^{n-k-1}+(1-p^{n-k-1})(1-(1-p)^{k(n-k-1)})-1\big].
\end{split}
\]
\end{lemma}

\begin{proof}
Let $U_k$ count the number of proper star $k$-separations $(a,S)$ for various $a\in V$ and $S\subset V\setminus\{a\}$ (with $|S|=k$).
Let $\wU{a}{S}$ be the random variable with value $1$ if $(a,S)$ is a proper star $k$-separation and value $0$ otherwise.

Of course the expectation $\Expect(\wU{a}{s})$ is the probability that $(a,S)$ is a proper star $k$-separation.
If $(a,S)$ is a proper star separation, then necessarily $a$ is not adjacent to any element of $S$.
This event has probabilty $(1-p)^k$.
This event is independent of the other aspects of the definition that we are about to describe, since we will not mention these edges again.

The next feature of the definition is that we must have every element of $V\drop (S\cup\{a\})$ either in $\lk(a)$ or not adjacent to any member of $S$.
Suppose $b\in V\drop (S\cup \{a\})$.
The event that $b$ is adjacent to $a$ has probability $p$; the event that $b$ is not adjacent to any element of $S\cup\{a\}$ is disjoint from this event and has probability $(1-p)^{k+1}$.
So the probability that $b$ is either adjacent to $a$ or not adjacent to any member of $S$ has probability $p+(1-p)^{k+1}$.
Since these events involve different edges for different choices of $b$, they are independent, and the event that every element of $V\drop (S\cup\{a\})$ is in $\lk(a)$ or not adjacent to anything in $S$ has probability
\[(p+(1-p)^{k+1})^{n-k-1}.\]

Next, we must exclude two events that are included in the previous event.
To fit the definition of a star separation, we must have that $V\drop S$ is not all of $\st(a)$.
This means that we are excluding the event that $a$ is adjacent to every element of $V\drop S$, which has probability $p^{n-k-1}$.
To ensure that $(S,a)$ is a proper star separation, we must exclude the event that $S$ is a separation.
Since we have already stipulated that there are no edges from $a$ to $S$, this is then the event that there are no edges from $V\drop (S\cup\{a\})$ to $S$.
This has probability $(1-p)^{k(n-k-1)}$.
The two events we have just described are independent, so by DeMorgan's law, the probability of either event happening is:
\[
1 - (1-p^{n-k-1})(1-(1-p)^{n(n-k-1)}).
\]

Putting this all together, we see that
\[
\begin{split}
\Expect(&\wU{a}{S})=\\
& (1-p)^k\big[(p+(1-p)^{k+1})^{n-k-1}+(1-p^{n-k-1})(1-(1-p)^{k(n-k-1)})-1\big].
\end{split}
\]
We note that there are $n$ possible choices for $a$ in $V$, and given a choice of $a$, there are $\binom{n-1}{k}$ choices for $S$.
Since $U_k$ is the sum of $\wU{a}{S}$ over all choices of $(a,S)$, the result follows from the linearity of expectations.
\end{proof}

Next we process this expression into a pair of more manageable bounds.
\begin{lemma}\label{le:newbounds}
We have the following bounds on $\Expect(U_k)$:
\[
\Expect(U_k)\leq n^2\binom{n-1}{k}(1-p)^{2k+1}(p+(1-p)^{k+1})^{n-k-2},
\]
and
\[
\Expect(U_k)\leq k n^2\binom{n-1}{k}(1-p)^{k+1}p^2(p+(1-p)^{k+1})^{n-k-2}.
\]
\end{lemma}

\begin{proof}
These bounds come from the expression in Lemma~\ref{le:countss} in essentially the same way.
We will use the following claim both times:
\begin{claim*}
Suppose $a$ and $b$ are real numbers with $0\leq a\leq b\leq 1$ and $n$ is a natural number.
Then
\[
a^n-b^n \leq na^{n-1}(b-a).
\]
\end{claim*}

The claim follows easily from a calculus argument: 
if 
\[f(a)=na^{n-1}(b-a)+b^n-a^n,\]
 then $f(b)=0$, but $f'(a)$ is negative for $a$ in $(0,b)$.

Now using the claim, we deduce the lemma from Lemma~\ref{le:countss}.
Since certainly, $1-p^{n-k-1}\leq 1$, we deduce first that 
\[
\Expect(U_k)\leq n\binom{n-1}{k}(1-p)^k((p+(1-p)^{k+1})^{n-k-1}-(1-p)^{k(n-k-1)}).
\]
Since $0\leq(1-p)^k\leq p+(1-p)^{k+1}\leq1$, the claim implies
\[
\begin{split}
(p&+(1-p)^{k+1})^{n-k-1}-(1-p)^{k(n-k-1)}
\\&\leq (n-k-1)(p+(1-p)^{k+1})^{n-k-2}(p+(1-p)^{k+1}-(1-p)^k)\\
&=(n-k-1)(p+(1-p)^{k+1})^{n-k-2}p(1-(1-p)^k).
\end{split}\]
Since $(1-p)^k\geq 1-kp$ for $p\in[0,1]$, it follows that
\[
\Expect(U_k)\leq kn(n-k-1)\binom{n-1}{k}p^2(1-p)^{k+1}(p+(1-p)^{k+1})^{n-k-2}.
\]

Second, since $1-(1-p)^{k(n-k-1)}<1$, we have
\[
\Expect(U_k)\leq n\binom{n-1}{k}(1-p)^k((p+(1-p)^{k+1})^{n-k-1}-p^{n-k-1}).
\]
Then since $0\leq p\leq p+(1-p)^{k+1}\leq 1$, the claim implies
\[
\Expect(U_k)\leq n(n-k-1)\binom{n-1}{k}(1-p)^{2k+1}(p+(1-p)^{k+1})^{n-k-1}.
\]
The lemma follows.
\end{proof}

\begin{proof}[Proof of Proposition~\ref{pr:smallstarseps}]
As usual, we split $\{(n,p(n))\}_n$ into subsequences satisfying stronger hypotheses that overlap.
Therefore it is enough to show the proposition for $p$ in several subcases with stronger hypotheses.
First we suppose that $p$ satisfies
\[2\frac{\log(n)+\omega(n)}{n} \leq p\leq 1-(k+1)\frac{\log(n)+\omega(n)}{n}.\]
We get the following bound from Lemma~\ref{le:countss}.
\[
\Expect(U_k)\leq n^{k+1}(p+(1-p)^{k+1})^{n-k-1}.
\]
(This is because $(1-p^{n-k-1})(1-(1-p)^{k(n-k-1)})-1\geq 0$.)
Then by Lemma~\ref{le:convexbound}, we have that $\Expect(U_k)$ converges to $0$ for any $p$ in this range.

Next we suppose that $q\to 0$ and $nq^2\to 0$, where $q=1-p$.
By the first bound from Lemma~\ref{le:newbounds}, we see that
\[\Expect(U_k)\leq n^{k+2}q^{2k+1}(1-q+q^{k+1})^{n-k-2}.\]
Since $(1-x)^m\sim e^{-mx}$ as $x\to 0$, this bound is asymptotically equivalent to 
\[n^{k+1}q^{2k+1}\exp(-nq+O(nq^2)+O(q)),\]
which is equivalent to
\[q^{k-1}(nq)^{k+2}e^{-nq}.\]
From calculus, we know that $x\mapsto x^{k+2}e^{-x}$ is bounded and tends to zero as $x\to \infty$ or $x\to 0$.
So if $k>1$ or $nq\to 0$ or $nq\to \infty$, we know that $\Expect(U_k)\to 0$.

Finally, we suppose that $p\to 0$, $np^2\to 0$, and 
\[p> \frac{\log(n)+(2/k)\log(\log(n))+\omega(n)}{n}\]
 for some $\omega(n)\to\infty$.
Using the second bound from Lemma~\ref{le:newbounds}, we have
\[\Expect(U_k)\leq n^{k+2}p^2(p+(1-p)^{k+1})^{n-k-2}.\]
By binomial expansion, this bound can be written as
\[n^{k+2}p^2(1-kp+O(p^2))^{n-k-2}.\]
This is asymptotically equivalent to 
\[n^{k+2}p^2 \exp(-knp + O(np^2)+O(p))\sim n^{k+2}p^2e^{-knp},\]
since $np^2,p\to 0$.
Then by our lower bound on $p$, we have
\[
\begin{split}
\Expect(U_k)&\lesssim n^{k+2}\cdot n^{-2}(\log(n)+\log(\log(n))+\omega(n))^2\cdot n^{-k}\log(n)^{-2}e^{-k\omega(n)}\\
&=\left(\frac{\log(n)+\log(\log(n))+\omega(n)}{\log(n)}\right)^2e^{-k\omega(n)}
\lesssim\omega(n)^2e^{-k\omega(n)}
\end{split}
\]
So $\Expect(U_k)\to 0$.

We have shown that under these hypotheses, there are a.a.s.\ no proper star $k$-separations.
However, Theorem~\ref{th:connectivitythreshold} implies that a.a.s.\ there are no separations for $p$ in this range.
So a.a.s.\ there are no star $k$-separations.
\end{proof}

\subsection{Summed counts of star separations}
\label{se:summedcounts}
So far, we have shown that for each $k$, $\Expect(U_k)\to 0$ in a certain range of probability values.
We would like to show that $\Expect(\sum_k U_k)\to 0$ for $p$ in a specific range.
Effectively, this requires commuting a limit and sum, which the Lebesgue dominated convergence theorem would allow.
To meet the hypotheses of this theorem, we must compute bounds on $\Expect(U_k)$.
We do this using calculus techniques, using two different bounds that are useful when $p$ is close to $1$ and when $p$ is close to $0$, respectively.
Unfortunately, showing these bounds carefully takes a fair amount of work.
We proceed to prove Theorem~\ref{th:masterstarcut}, and then we prove the bounds on $\Expect(U_k)$.

\begin{proof}[Proof of Theorem~\ref{th:masterstarcut}]
We break the sequence $\{(n,p(n))\}_n$ into two subsequences, where one satisfies $p\leq 2/5$, and the other satisfies $p\geq 2/5$.
We prove the theorem for each subsequence; then it follows that the theorem is true for the original sequence.

We suppose the first hypothesis of the theorem, that $p>n^{-1}(\log(n)+\log(\log(n))+\omega(n))$. 
We note that every star $1$-separation consists of a non-adjacent domination pair.
Then Theorem~\ref{th:dominationnonexistence} shows that a.a.s., there are no star $1$-separations if $p\leq 2/5$.
If $p\geq 2/5$, then Proposition~\ref{pr:smallstarseps} includes the fact that there are a.a.s.\ no star $1$-separations, provided that the second hypothesis also holds.
Then quoting Proposition~\ref{pr:smallstarseps}, we have that
\[\lim_{n\to\infty} \Expect(U_k)=0,\]
for any fixed $k$, if $k\geq 2$ (and $p$ satisfies the first hypothesis) or if $k=1$ and $p$ satisfies both hypotheses of the theorem.

Proposition~\ref{pr:highpbound} below states that if $p\geq 2/5$, then there is a nonnegative sequence $\{a_k\}_k$ with $\Expect(U_k)\leq a_k$ for all $n$ and for all $k$ with $2k\leq n$, and such that $\sum a_k<\infty$.
Proposition~\ref{pr:lowpbound} is the same statement in the case that $p\leq 2/5$.
Then in either case, the Lebesgue dominated convergence theorem (see, e.g. Rudin~\cite[p.26]{Rudin}) applies.
Therefore if we assume both hypotheses, we have 
\begin{equation}\label{eq:ldc}
\lim_{n\to\infty}\sum_{k=1}^{\lfloor n/2\rfloor}\Expect(U_k)=\sum_{k=1}^{\infty}\lim_{n\to\infty}\Expect(U_k)=0.
\end{equation}
Of course, by linearity of expectations, we have
\[\sum_{k=1}^{\lfloor n/2\rfloor}\Expect(U_k)=\Expect\left(\sum_{k=1}^{\lfloor n/2\rfloor}U_k\right).\]
The random variable on the right will be zero only if $\Gamma\in\gnp$ has no proper star-separations: if $a\in\Gamma$ is a star-cut-vertex, then some component of $\Gamma\setminus\st(a)$ has less than $n/2$ vertices.
Then by Equation~\eqref{eq:ldc} there are a.a.s.\ no proper star separations.
Since Theorem~\ref{th:connectivitythreshold} implies there are a.a.s.\ no separations, we know there are a.a.s.\ no star-cut-vertices.

Similarly, if we assume only the first hypothesis, Equation~\eqref{eq:ldc} will be true if we sum from $k=2$ to $k=\lfloor n/2\rfloor$ on both sides, instead of starting at $k=1$.
Then we deduce that
\[\lim_{n\to\infty}\Expect\left(\sum_{k=2}^{\lfloor n/2\rfloor} U_k\right)=0.\]
Of course, this random variable will be zero only if there are no proper star $k$-separations for $2\leq k<n/2$.
If the star of a star-cut-vertex has more than one complementary component with at least two vertices, then it will have a complementary component with at least two and fewer than $n/2$ vertices.
The second statement in the theorem follows.
\end{proof}

Now we bound $\Expect(U_k)$.  
The choice of $2/5$ below is somewhat arbitrary.

\begin{proposition}\label{pr:highpbound}
There is a sequence of positive numbers $a_k$ such that $\sum_{k=1}^\infty a_k<\infty$ and $\Expect(U_k)\leq a_k$ for any $n\geq 2k$ and any $p\geq 2/5$.
\end{proposition}

\begin{proof}
We write the first bound on $\Expect(U_k)$ from Lemma~\ref{le:newbounds} in terms of $q=1-p$:
\[\Expect(U_k)\leq n\binom{n-1}{k} (n-k-1)q^{2k+1}(1-q+q^{k+1})^{n-k-2}.\]
In terms of $q$, our hypothesis is that $q\leq 3/5$.
We then use a bound on $\binom{n-1}{k}$ that can be derived by considering $\log \binom{n-1}{k}$ as a Riemann sum approximation to an integral of a continuous function:
\[\binom{n-1}{k}\leq \frac{n^n}{(n-k)^{n-k}k^k}.\]
We set
\[F(k,n,q)= \frac{n^{n}}{(n-k)^{n-k}k^k}n^2q^{2k+1}(1-q+q^{k+1})^{n-k-2};\]
then
\[\Expect(U_k)\leq F(k,n,q).\]
We find the bounding sequence by using vector calculus to find critical points for $F(k,n,q)$ for fixed $k$ and for $(n,q)$ in the region $[2k,\infty)\times[0,3/5]$.
\begin{claim*}
For large enough $k$, the partial derivative $\partial \log\circ F/\partial n$ is never zero on the vertical ray $(n,q)\in[2k,\infty)\times\{3/5\}$.
\end{claim*}
To show this, we consider
\[\frac{\partial \log\circ F}{\partial n}(k,n,q)=\frac{2}{n}+\log(\frac{n}{n-k})+\log(1-q+q^{k+1}).\]
This is zero if and only if
\[e^{-2/n}(n-k)+n(-1+q-q^{k+1})=0.\]
Define $f_k(n,q)=e^{-2/n}(n-k)+n(-1+q-q^{k+1})$.
Next we define $g_k(n,q)=-(k+2)+n(q-q^{k+1})$.
We use the following bound from calculus: \[|e^{-2/n}-(1-2/n)|\leq 2/n^2.\]
This implies that
\begin{equation*}
f_k(n,q)-g_k(n,q)\leq 2/n+k(2/n^2-2/n)\leq 2/k,\end{equation*}
since $n\geq 2k$.

The function $g_k(n,q)$ is chosen so that we can solve $g_k(n,q)=0$ for $n$ easily:
\begin{equation*}
g_k(n,q)=0 \text{ if and only if } n = \frac{k+2}{q-q^{k+1}}.
\end{equation*}
To show the claim, we show that $f_k(n,q)$ is not zero on the vertical ray.
Note that $g_k(n,3/5)\geq -(k+2)+2k((3/5)-(3/5)^{k+1})$ on this ray.
Since this bound is asymptotically equivalent to $(1/5)k$ we see that for large $k$, $g_k(n,3/5)>2/k$ on the ray.
This implies that $f_k(n,3/5)$ is not zero when $n>2k$ and $k$ is large enough, proving the claim.

\begin{claim*}
For large enough $k$, the function $\log\circ F$ never has zero gradient on $[2k,\infty)\times[0,3/5]$.
\end{claim*}
Note that
\begin{equation}\label{eq:pFpq}
\frac{\partial \log\circ F}{\partial q}(k,n,q)=\frac{2k+1}{q}+(n-k-2)\frac{-1+(k+1)q^k}{1-q+q^{k+1}}.
\end{equation}
We suppose that this partial derivative is $0$ and $f_k(n,q)$ is zero as well.
If $f_k(n,q)$ is zero, then $g_k(n,q)=\epsilon$ for some $\epsilon=\epsilon(n,q,k)$ with $|\epsilon|>2/k$.
Then 
\[q-q^{k+1}=\frac{k+2-\epsilon}{n}.\]
Assuming that $\frac{\partial \log\circ F}{\partial q}=0$, then substituting this for one instance of $q-q^{k+1}$ in the expression in Equation~\eqref{eq:pFpq}, we get the equation
\[\frac{2k+1}{q}+\frac{n-k-2}{n-k-2+\epsilon}\cdot n(-1+(k+1)q^k)=0.\]
Next we use the substitution $n=(k+2-\epsilon)/(q-q^{k+1})$ on one of the instances of $n$ to get
\[\frac{2k+1}{q}+\frac{n-k-2}{n-k-2+\epsilon}\cdot\frac{k+2-\epsilon}{q-q^{k+1}}\cdot (-1+(k+1)q^k)=0.\]
We write $c=\frac{n-k-2}{n-k-2+\epsilon}$.
Note that $c$ is close to $1$.
From here, it is straightforward to solve for $q^k$:
\[q^k=\frac{(2-c)k+1-c(2-\epsilon)}{-ck^2+(2-x(3-\epsilon))k+1-c(2-\epsilon)}.\]
The right side of this equation is negative, so it has no real solution for $q$ if $k$ is even.
Taking the limit as $k\to\infty$ on odd $k$, we see that $q\to -1$.
This implies that for large enough $k$, we never have both $f_k(n,q)=0$ and  $\frac{\partial \log\circ F}{\partial q}=0$; this proves the claim.

\begin{claim*}
For large enough $k$, the maximum of $F(k,n,q)$ for $(n,q)\in[2k,\infty)\times[0,3/5]$ is realized at $(n,q)=(2k,3/5)$.
\end{claim*}
We have shown that the gradient of $\log\circ F$ never vanishes on $[2k,\infty)\times[0,3/5]$, and that the partial with respect to $n$ never vanishes on the left boundary of the region.
Note that $F$ is zero on the right boundary.
So it is enough to show that $F$ is increasing in $q$ for $q<3/5$ when $n=2k$.
It is straightforward to see from Equation~\eqref{eq:pFpq} that:
\[\frac{\partial \log\circ F}{\partial q}\geq\frac{2k}{q}-\frac{k-2}{1-q}.\]
However, $q\mapsto \frac{2k}{q}-\frac{k-2}{1-q}$ is plainly decreasing in $q$;
evaluating it at $3/5$, we see that $\frac{\partial \log\circ F}{\partial q}\geq (5/6)k+5>0$.
This proves this last claim.

Then to prove the lemma, we note the following:
\[\lim_{k\to\infty}\frac{F(k+1,2k+2,3/5)}{F(k,2k,3/5)}=\frac{72}{125}<1.\]
So the maximum value of $\Expect(U_k)$ for $q$ in this range is eventually bounded above by an exponentially decaying function of $k$.
\end{proof}

\begin{proposition}\label{pr:lowpbound}
There is a sequence of positive numbers $a_k$ such that $\sum_{k=1}^\infty a_k<\infty$ and 
$\Expect(U_k)\leq a_k$
for any $n\geq 2k$ and any $p$ satisfying
\[\frac{\log(n)+\log(\log(n))}{n} \leq p \leq 2/5.\]
\end{proposition}

\begin{proof}
Define $G(k,n,p)$ by 
\[G(k,n,p)=\frac{n^n}{(n-k)^{n-k}k^{k-1}}n^2p^2(1-p)^{k+1}(p+(1-p)^{k+1})^{n-k-2}.\]
Then by Lemma~\ref{le:newbounds} (the second bound) and the bound on $\binom{n-1}{k}$ from the previous proposition, we know that 
\[G(k,n,p)\geq \Expect(U_k).\]
To bound $\Expect(U_k)$ by a function of $k$, we will show that for sufficiently large fixed values of $k$, the maximum of $G$ on the region
\[R=\{(n,p)| n\geq 2k, \frac{\log(n)+\log(\log(n))}{n} \leq p \leq 2/5\}\]
is at one of the corners of $R$.

\begin{claim}
For large enough $k$, the gradient of $\log\circ G$ is never zero on $R$.
\end{claim}
Since $p\mapsto 1-p-(1-p)^{k+1}$ is concave down, it is easy to verify that for large enough $k$ and $p<2/5$, we have
\[1-p-(1-p)^{k+1}\geq \min(\frac{(k+2)p}{2},\frac{k+2}{2k}).\]
We use this inequality in a bound on
\[\frac{\partial \log\circ G}{\partial n}(k,n,p)=\frac{2}{n}+\log(1+\frac{k}{n-k})+\log(p+(1-p)^{k+1}).\]
It is immediate that
\[\frac{\partial \log\circ  G}{\partial n}(k,n,p)\leq \frac{2}{n}+\frac{k}{n-k}-1+p+(1-p)^{k+1},\]
from which we deduce
\[\frac{\partial \log\circ  G}{\partial n}(k,n,p)\leq \frac{2}{n}+\frac{k}{n-k} - \min(\frac{(k+2)p}{2},\frac{k+2}{2k}).\]
If $\frac{k+2}{2k}$ is the smaller quantity, then it is straightforward to check that this bound is negative (using $n\geq 2k$).
If $\frac{(k+2)p}{2}$ is the smaller quantity, then using $p\geq n^{-1}(\log(n)+\log(\log(n)))$ and $n\geq 2k$, it is also routine to check that the bound is negative for large $k$.
Then $\frac{\partial \log\circ G}{\partial n}$ is always negative on $R$, so that the gradient is never zero on $R$.

Next we deal with the sloping boundary of $R$.
\begin{claim}
For all $k$ large enough and $n\geq 2k$, the function $G(k,n,p(n))$ is decreasing in $n$, where
\[p=p(n)=\frac{\log(n)+\log(\log(n))}{n}.\]
\end{claim}
First we compute the logarithmic partial derivative of $G(k,n,p(n))$ with respect to $n$:
\[\begin{split}
\frac{\partial}{\partial n}\log(G(k,n,p(n)))&= \log(\frac{n}{n-k})+\frac{2}{n}+\frac{2p'}{p}-\frac{kp'}{1-p}
\\&\quad
+\log(p+(1-p)^{k+1}))
\\&\quad
 + p'(n-k-2)\frac{1-(k+1)(1-p)^{k}}{p+(1-p)^{k+1}}.
\end{split}
\]
In computing upper bounds on this expression, we will freely assume that $k$ is large, and we will always assume that $n\geq 2k$.

Since $1-(k+1)^{-1/k}$ goes to $0$ more slowly than $p(2k)$ as $k\to\infty$, we may assume that $1-(k+1)(1-p)^{k}$ is negative.
We may assume that $p'$ is negative.
Further, note that $p+(1-p)^{k+1}\geq (1-p)^k$.
Then
\[p'(n-k-2)\frac{1-(k+1)(1-p)^{k}}{p+(1-p)^{k+1}}\leq (n-k-2)p'((1-p)^{-k}-1-k).\]
By logarithm rules and the inequality $\log(1+x)\leq x$, we may deduce
\[
\log(\frac{n}{n-k})\leq \frac{k}{n-k}\text{ and }\log(p+(1-p)^{k+1})\leq p((1-p)^{-k-1}-1-k).\]
Then we have
\begin{equation}\label{eq:derivbound}
\begin{split}
\frac{\partial}{\partial n}\log(G(k,n,p(n)))&\leq \frac{k}{n-k}+\frac{2}{n}+\frac{2p'}{p}-\frac{kp'}{1-p}
\\&\quad
+p((1-p)^{-k-1}-1-k)
\\&\quad
 + (n-k-2)p'((1-p)^{-k}-1-k).
\end{split}
\end{equation}

To process this expression, we start combining terms.
First of all, it is straightforward to show
\[\frac{2}{n}+\frac{2p'}{p}=2\frac{\log(n)+1}{\log(n)(\log(n)+\log(\log(n)))}\leq \frac{2}{\log(n)}.\]
Next, we note
\[\frac{k}{n-k}-k(p+np')=\frac{k^2}{n(n-k)}-\frac{k}{n\log(n)}.\]
Since
\[p'\leq -\frac{\log(n)}{n^2},\]
we have 
\[\frac{k}{n-k}-k(p+(n-k-2)p')\leq \frac{k^2}{n(n-k)}-\frac{k}{n\log(n)}-\frac{k(k+2)\log(n)}{n^2}.\]
Next we note that 
\[\frac{-kp'}{1-p}\leq \frac{2k\log(n)}{n^2}.\]
Since the $-k(k+2)\log(n)/n^2$ dominates the positive terms, we may deduce that for any positive constant strictly less than $1$, say $1/2$, we have
\begin{equation}
\label{eq:mainbound}
\frac{k}{n-k}-k(p+(n-k-2)p')+\frac{2}{n}+\frac{2p'}{p}-\frac{kp'}{1-p}\leq -\frac{k}{n\log(n)}-\big(\frac{1}{2}\big)\frac{k^2\log(n)}{n^2}.
\end{equation}
The remaining terms in our bound from Equation~\eqref{eq:derivbound} may be written as:
\begin{equation}\label{eq:extraterms}(p+(n-k-2)p')((1-p)^{-k-1}-1)-pp'(n-k-2)(1-p)^{-k-1}.\end{equation}
Since we assume that $p<1/2$, we know $-2\log(2)p<\log(1-p)$.
Since $(1-p)^{-k-1}$ is decreasing in $n$ for fixed $k$, we may get an upper bound by plugging in for $n=2k$.
Then
\[(1-p)^{-k-1}\leq (2k\log(2k))^{\log(2)\frac{k+1}{k}}\leq 2k^{\frac{3}{4}}.\]
This gives us an immediate upper bound on the second term from Equation~\eqref{eq:extraterms} as follows, using the obvious bound $p\leq 2\log(n)/n$:
\[-pp'(n-k-2)(1-p)^{-k-1}\leq \frac{8k^{\frac{3}{4}}(\log(n))^2}{n^2}.\]
Next we bound the first term in Equation~\eqref{eq:extraterms}.
Since $(1-p)^{k+1}\geq 1-(k+1)p$, we deduce that 
\[(1-p)^{-k-1}-1\leq (k+1)p(1-p)^{-k-1}\leq \frac{4(k+1)k^{\frac{3}{4}}\log(n)}{n}.\]
Then
\[
\begin{split}
(p+(n-k-2)p')((1-p)^{-k-1}-1) &\leq
\frac{8(k+1)k^{\frac{3}{4}}\log(n)}{n^2} \\
&\quad +\frac{8(k+1)(k+2)k^{\frac{3}{4}}\log(n)^2}{n^3}.
\end{split}
\]
Using our assumption that $n\geq 2k$, we may then deduce that
\[
\begin{split}
(p+(n-k-2)p')((1-p)^{-k-1}-1)& \leq \frac{8(k+1)k^{\frac{3}{4}}\log(n)}{n^2} \\
&\quad +\frac{8(k+1)\log(k)k^{\frac{3}{4}}\log(n)}{n^2}.
\end{split}
\]
Then the $-k^2\log(n)/n^2$ term from Equation~\eqref{eq:mainbound} eventually (for large $k$ and all $n\geq 2k$) dominates everything from Equation~\eqref{eq:extraterms}, and the claim follows.

\begin{claim}
For large fixed $k$, the minimum value of $G(k,n,p)$ in $R$ is realized at one of the corners of $R$: 
\[(n,p)=(2k,2/5)\text{ or }(n,p)=(2k,(2k)^{-1}(\log(2k)+\log(\log(2k)))).\]
\end{claim}
We have shown that the gradient of $\log\circ G$ never vanishes on $R$ and $G$ is decreasing on the sloping boundary of $R$.
In showing that the gradient never vanishes, we showed that $\frac{\partial\log\circ G}{\partial n}$ is negative on the boundary segment $p=2/5$.
This means that the maximum cannot occur on this boundary segment (away from the corner).
Now we consider the boundary segment where $n=2k$.
We consider
\[\frac{\partial \log\circ G}{\partial p}=\frac{2}{p}-\frac{k}{1-p}+(n-k-2)\frac{1-(k+1)(1-p)^k}{p+(1-p)^{k+1}}.\]
When $n=2k$, this is the function
\[p\mapsto \frac{2}{p}-\frac{k}{1-p}+(k-2)\frac{1-(k+1)(1-p)^k}{p+(1-p)^{k+1}}.,\]
where $(2k)^{-1}(\log(2k)+\log(\log(2k)))\leq p\leq 2/5$.
We define a function $f\co ((2k)^{-1}(\log(2k)+\log(\log(2k))),2/5)\to\R$ to be this function with denominators cleared:
\[f(p)= (2(1-p)-kp)(p+(1-p)^{k+1})+(k-2)p(1-p)(1-(k+1)(1-p)^k).\]
Computations show:
\[f''(p)=-4k+(1-p)^{k-2}g(k,p),\]
where $g$ is some polynomial in $k$ and $p$,
and
\[f'''(p)=(1-p)^{k-3}(2 k-5 k^3-3 k^4+(-2 k+7 k^3+6 k^4+k^5) p-(2 k^3+3 k^4+k^5) p^2).\]
A computation shows that $f'''(p)$ is positive if 
\[\frac{-2+2 k+3 k^2}{k^2 (2+k)} < p < 1,\]
so for large $k$, $f'''(p)$ is positive on the entire domain of $f$.
Then $f''(p)\leq f''(2/5)$, for $p$ in the domain of $f$.
It is easy to see that $f''(2/5)\sim -4k$ as $k\to\infty$, so that $f''(2/5)$ is negative for large $k$ and $p$ in the domain of $f$.
Then $f$ is concave down.
Since $f(2/5)\sim \frac{2}{25}k$ as $k\to\infty$, this means that $f$ changes sign at most once.
Then $G$ is decreasing-then-increasing as $p$ increases along the boundary segment of $R$ with $n=2k$.
In particular, this proves the claim.

Finally, we can prove the proposition.
We note that
\[\lim_{k\to\infty}\frac{G(k+1,2k+2,2/5)}{G(k,2k,2/5)}=24/25<1,\]
and 
\[\lim_{k\to\infty}\frac{G(k+1,2k+2,p(2k+2))}{G(k,2k,p(2k))}=0,\]
where $p(n)=n^{-1}(\log(n)+\log(\log(n)))$.
Then the sequence of values of $G$ at each corner of $R$ is eventually dominated by an exponentially decreasing function.
\end{proof}

\section*{Acknowledgments}
I was an undergraduate participant in the 2001 SUNY Potsdam--Clarkson summer REU, where I was involved with one research project on random graphs and another project on graph products of groups.
This REU was apt preparation for a project like this paper, and I am grateful to Christino Tamon and Kazem Mahdavi for their supervisory roles in those projects.
Thanks to Hanna Bennett for conversations and notes on an earlier version of this paper.
I am grateful to Niranjan Balachandran, Ruth Charney, Helge Kr\"uger and Michael Mendenhall for conversations.

\noindent
California Institute of Technology\\
Mathematics 253-37\\
Pasadena, Ca 91125\\
E-mail: {\tt mattday@caltech.edu}
\medskip

\end{document}